\pdfoutput=1
\documentclass[12pt,reqno]{amsart}
\usepackage[lmargin=1.35in,rmargin=1.35in,tmargin=1.25in,bmargin=1.25in]{geometry}

\usepackage{yfonts}
\usepackage{amsthm,amsfonts,amssymb}

\usepackage{enumitem}
\usepackage[colorlinks]{hyperref}    
\hypersetup{linkcolor=blue, citecolor=blue}

\usepackage[mathcal]{euscript}
\usepackage{booktabs}

\renewcommand{\epsilon}{\varepsilon}

\let\originalleft\left
\let\originalright\right
\renewcommand{\left}{\mathopen{}\mathclose\bgroup\originalleft}
\renewcommand{\right}{\aftergroup\egroup\originalright}

\newcommand{\fr}[2]{\frac{#1}{#2}}
\newcommand{\nfr}[2]{#1/#2}
\newcommand{\tf}[2]{\tfrac{#1}{#2}}

\newtheorem{theorem}{Theorem}

\newtheorem{lemma}[theorem]{Lemma}
\newtheorem{corollary}[theorem]{Corollary}

\theoremstyle{definition}
\newtheorem{remark}[theorem]{Remark}

\numberwithin{theorem}{section}
\numberwithin{equation}{section}

\newcommand{\<}{\begin{equation}}
\renewcommand{\>}{\end{equation}}

\newcommand{\lf}{\left}

\newcommand{\rh}{\right}

\newcommand{\rt}[1]{\sqrt{#1}}

\newcommand{\mmod}[1]{\,\,(\mathrm{mod}\,\,#1)}

\newcommand{\cc}{\mathbb{C}}
\newcommand{\nn}{\mathbb{N}}
\newcommand{\zz}{\mathbb{Z}}

\newcommand{\fp}{\mathfrak{p}}
\newcommand{\fX}{\mathfrak{X}}
\newcommand{\fY}{\mathfrak{Y}}

\newcommand{\xa}{\alpha}
\newcommand{\xb}{\beta}
\newcommand{\xc}{\chi}
\newcommand{\xd}{\delta}

\newcommand{\xg}{\gamma}

\newcommand{\xl}{\lambda}

\newcommand{\xF}{\Phi}
\newcommand{\xL}{\Lambda}

\newcommand{\cR}{\mathcal{R}}

\usepackage{fancyhdr}
\usepackage[initials,nobysame]{amsrefs}
\usepackage[foot]{amsaddr}

\makeatletter
\renewcommand{\email}[2][]{%
  \ifx\emails\@empty\relax\else{\g@addto@macro\emails{,\space}}\fi%
  \@ifnotempty{#1}{\g@addto@macro\emails{\textrm{(#1)}\space}}%
  \g@addto@macro\emails{#2}%
}
\makeatother

\makeatletter
\def\section{%
    \@startsection{section}{1}%
    \z@{.7\linespacing\@plus\linespacing}{.5\linespacing}%
    {\normalfont\large\bfseries}%
}
\makeatother

\makeatletter
\def\@seccntformat#1{%
  \protect\textup{\protect\@secnumfont
    \csname the#1\endcsname
\space\space
  }%
}
\makeatother

\AtBeginDocument{%
   \def\MR#1{}
}

\begin{document}

\title{Vanishing coefficients in two $q$-series related to Legendre-signed partitions} 
\author[T.~Daniels]{Taylor Daniels}
\address{Purdue University, 150 N University St, W Lafayette, IN 47907}
\email{daniel84@purdue.edu}
\begin{abstract}
    We demonstrate some 10-periodic properties of the coefficients of two $q$-series related to partition numbers signed by the Legendre symbol $(\frac{a}{5})$ and to the Rogers-Ramanujan continued fraction $R(q)$.
\end{abstract}
\maketitle

\section{Introduction}
The \emph{partitions} of a given $n \in \nn$ are the tuples $(a_1,\ldots,a_k)$ of positive integers such that $a_1+\cdots+a_k=n$ and $a_1 \geq \cdots \geq a_k$. In this paper we consider two $q$-series arising naturally from the notion of partitions ``signed'' by the Legendre symbols $\xc_p(n) := (\frac{n}{p})$ associated to odd primes $p$; namely $\xc_p(n)$ is $1$ (or $-1$) when $n$ is (or is not) a quadratic residue modulo $p$, and $\xc_p(n)$ is $0$ when $p \mid n$. For $n \in \nn$ and any partition $\pi = (a_1, a_2, \ldots, a_k)$ of $n$, let
    \begin{align*}
        \xc_p(\pi) &:= \xc_p(a_1)\xc_p(a_2)\cdots\xc_p(a_k) \qquad\text{and}\qquad \xc_p^\dagger(\pi) := (-1)^k\xc_p(\pi).
    \end{align*}
With this notation we define the integers $\mathfrak{p}(n,\chi_p)$ and $\mathfrak{p}(n,\chi_p^\dag)$ via
    \<
        \mathfrak{p}(n,\chi_p) = \sum_{\pi\in\Pi[n]} \xc_p(\pi) \qquad\text{and}\qquad \mathfrak{p}(n,\chi_p^\dagger) = \sum_{\pi\in\Pi[n]} \xc_p^\dagger(\pi),
    \>
where $\Pi[n]$ is the set of all partitions of $n$.

The quantities $\mathfrak{p}(n,\chi_5)$ are the \emph{Legendre-signed partition numbers} (associated to $5$) as introduced and studied in \cite{daniels2024legendre}. In \cite{daniels2024legendre}*{Thm.~1.7}, it is shown that as $n\to\infty$, one has
    \<
        \label{eq:int:P5Form}
		\mathfrak{p}(n,\chi_5) = \big(\tf{3+\sqrt{5}}{960}\big)^{\fr14} n^{-\fr34}\exp\!\Big(\pi\rt{\tf{8}{15}n}\,\Big) \lf[ \mathfrak{S}(n) + O(n^{-\fr15}) \rh],
    \>
    where
    \[
        \mathfrak{S}(n) := 1 + (-1)^n \Big(\fr{3-\sqrt{5}}{2}\Big) + \sqrt{2(5-\sqrt{5})}\cos\!\Big(\fr{2\pi n}{5}-\fr{\pi}{10}\Big).
    \]
Because $\mathfrak{S}(n) = 0$ for all $n \equiv 2 \mmod{10}$ and $\mathfrak{S}(n) \neq 0$ for all other $n$, formula \eqref{eq:int:P5Form} suggests that $\mathfrak{p}(n,\chi_5) = 0$ for all sufficiently large $n \equiv 2 \mmod{10}$. Confirming both this periodic vanishing of $\mathfrak{p}(n,\chi_5)$ and a similar periodic vanishing of $\mathfrak{p}(n,\chi_5^\dagger)$ is the primary objective of this paper.

\begin{theorem}
    \label{thm:X2Y6}
    One has $\mathfrak{p}(10j+2,\chi_5) = 0$ and $\mathfrak{p}(10j+6,\chi_5^\dagger)=0$ for all $j \geq 0$.
\end{theorem}

Considering Theorem \ref{thm:X2Y6}, we say that $\fp(n,\chi_5)$ \emph{vanishes on the arithmetic progression} $(10j+2)_{j\geq 0}$.
This vanishing is surprising in light of the following result from \cite{daniels2024legendre}.

\begin{theorem}[\cite{daniels2024legendre}*{Thm.~1.10}]
    \label{thm:novanishing}
    If $p$ is an odd prime such that $p\neq 5$ and $p \not\equiv 1 \mmod{8}$, then the sequence $(\fp(n,\xc_p))_\nn$ does not vanish on any arithmetic progression.
\end{theorem}

We remark that Theorem \ref{thm:novanishing} is established by showing that the asymptotic formulae for $\fp(n,\chi_p)$ (analogous to formula \eqref{eq:int:P5Form}) are dominated by a single positive exponential term, and thus $\fp(n,\chi_p) \to \infty$ for the $p$ in the theorem.
In addition to the vanishings of Theorem \ref{thm:X2Y6}, we establish two further relations between $\mathfrak{p}(n,\chi_5)$ and $\mathfrak{p}(n,\chi_5^\dagger)$.

\begin{theorem}
    \label{thm:XY8-XY0}
    One has $\mathfrak{p}(10j,\chi_5^\dagger) = \mathfrak{p}(10j,\chi_5)$ and $\mathfrak{p}(10j+8,\chi_5^\dagger) = -\mathfrak{p}(10j+8,\chi_5)$ for all $j \geq 0$. 
\end{theorem}

We consider $\mathfrak{p}(n,\chi_5)$ and $\mathfrak{p}(n,\chi_5^\dagger)$ primarily through the lens of $q$-\emph{series}, that is, series of the form $\sum_{n=0}^\infty a_nq^n$ defined for $|q|<1$. Using the notations
    \[
        (z;q)_\infty = \prod_{n=0}^{\infty} (1-zq^n)  \qquad\text{and}\qquad (z_1,\ldots,z_m;q)_\infty = (z_1;q)_\infty\cdots(z_m;q)_\infty,
    \]
where $|q|<1$ and $z\in\cc$, we define 
    \<
        \label{eq:XYdef}
        X(q) := (q,-q^2,-q^3,q^4;q^5)_\infty \qquad\text{and}\qquad Y(q) := (-q,q^2,q^3,-q^4;q^5)_\infty
    \>
so that
    \<
        \label{eq:qseriesXY}
        \fr{1}{X(q)} = \sum_{n=0}^\infty \mathfrak{p}(n,\chi_5) q^n
        \qquad\text{and}\qquad
        \fr{1}{Y(q)} = \sum_{n=0}^\infty \mathfrak{p}(n,\chi_5^\dagger) q^n.
    \>

The key to our analyses of the series in \eqref{eq:qseriesXY} is the recognition that, since
    \[
        X(q)Y(q) = (q^2,q^4,q^6,q^8;q^{10})_\infty = \fr{(q^2;q^2)_\infty}{(q^{10};q^{10})_\infty},
    \]
we have the useful pair of relations
    \<
        \label{eq:XYrels}
        \fr{1}{X(q)} = \fr{Y(q)(q^{10};q^{10})_\infty}{(q^2;q^2)_\infty} \qquad\text{and}\qquad \fr{1}{Y(q)} = \fr{X(q)(q^{10};q^{10})_\infty}{(q^2;q^2)_\infty}.
    \>
In particular, relations \eqref{eq:XYrels} allow us to leverage results on the well-studied function $(q;q)_\infty^{-1}$ into results on $X(q)^{-1}$ and $Y(q)^{-1}$.\newline

\noindent\textbf{Acknowledgements} The author would like to thank Prof.~Trevor Wooley for suggesting this research, and for financially supporting this research in part using Zoltners Professor start-up funding. To assist any readers wishing to verify or better understand our derivations, a Mathematica notebook documenting and implementing our computations is publicly available on the author's personal website. In addition we thank the referees for helpful feedback during publication.

\section{Preliminaries}

The variable $q$ is always a complex number with $|q|<1$, and $q$-\emph{series} are simply elements of $\zz[\hspace{-0.1em}[q]\hspace{-0.1em}]$. Euler's product $(q;q)_\infty$ is denoted by $f = f(q)$, and for integer $m > 0$ we abbreviate $f(q^m)$ as $f_m$. We recall the well-known Jacobi triple product formula
\begin{align}
    \label{eq:JTP}
    (\pm q^{a-b}, \pm q^{a+b},q^{2a};q^{2a})_\infty &= \sum_{m=-\infty}^\infty (\mp 1)^m q^{am^2+bm}.
\end{align}
For any $q$-series $g(q) = \sum_{n=0}^\infty a_n q^n$ we use the notation
    \<
        [g]_{[r]} := \sum_{j=0}^\infty a_{10j+r} q^{10j} \qquad (0 \leq r \leq 9)
    \>
to denote the \emph{type}-10 \emph{components} of $g$, and we observe that
    \begin{align*}
        [g(q)+h(q)]_{[r]} &= [g(q)]_{[r]} + [h(q)]_{[r]}, \\
        [g(q)h(q^{10})]_{[r]} &= h(q^{10})[g(q)]_{[r]}
    \end{align*}
for $0 \leq r \leq 9$ and any $q$-series $h(q)$. By \eqref{eq:qseriesXY} then, Theorem \ref{thm:X2Y6} states that
    \[ 
        [X^{-1}]_{[2]} = [Y^{-1}]_{[6]} = 0,
    \] 
and Theorem \ref{thm:XY8-XY0} states that 
    \[ 
        [Y^{-1}]_{[0]} = [X^{-1}]_{[0]} \qquad\text{and}\qquad [Y^{-1}]_{[8]} = -[X^{-1}]_{[8]}.
    \] 

An $m$--\emph{dissection} of a $q$-series $g(q)$ is any identity of the form
    \[
        g(q) = q^{0}g_0(q^{m}) + q^{1}g_1(q^{m}) + \cdots + q^{m-1}g_{m-1}(q^{m})
    \]
where $g_0, g_1, \ldots, g_{m-1}$ are all $q$-series; we are thus primarily interested in 10-dissections of $X^{-1}$ and $Y^{-1}$. Finally we recall the Rogers-Ramanujan continued fraction
    \<
        \label{eq:R}
        R(q) = \fr{(q,q^4;q^5)_\infty}{(q^2,q^3;q^5)_\infty},
    \>
and since $(q^a;q^b)_\infty(-q^a;q^b)_\infty = (q^{2a};q^{2b})_\infty$ for $a,b \in \nn$, we see that $X(q)$ and $Y(q)$ are related to $R(q)$ via the equalities
    \<
        \label{eq:XYR}
        X(q) = (q^4,q^6;q^{10})_\infty R(q) \qquad\text{and}\qquad Y(q) = (q^2,q^8;q^{10})_\infty R(q)^{-1}.
    \>

\section{The 10-dissections}
\label{sec:dissections}
After multiplying both sides by $f_{10}^3$, relations \eqref{eq:XYrels} state that
    \<
        \label{eq:XYfracs}
        \fr{f_{10}^3}{X} = \fr{Yf_{10}^4}{f_2} \qquad\text{and}\qquad \fr{f_{10}^3}{Y} = \fr{Xf_{10}^4}{f_2},
    \>
which, since $[f_{10}g]_{[r]} = f_{10}[g]_{[r]}$ for $0 \leq r \leq 9$ and any $q$-series $g(q)$, implies that
    \[
        f_{10}^3\left[\fr{1}{X}\right]_{[r]} = \left[\fr{Yf_{10}^4}{f_2}\right]_{[r]} \qquad\text{and}\qquad f_{10}^3\left[\fr{1}{Y}\right]_{[r]} = \left[\fr{Xf_{10}^4}{f_2}\right]_{[r]} \qquad (0 \leq r \leq 9).
    \]
Our plan, therefore, is to find suitable 10-dissections of $1/f_2$, $Xf_{10}^4$, and $Yf_{10}^4$, and use said dissections to draw conclusions about the type-10 components of $X^{-1}$ and $Y^{-1}$.

For the 10-dissection of $1/f_2$, we recall \cite{hirschhorn:powerofq}*{eqn.~(8.4.4)} the 5-dissection
    \<
    \label{eq:phiR}
    \begin{aligned}
        \fr{1}{f} = \fr{f_{25}^5}{f_{5}^6} \Big( R(q^5)^{-4} &+qR(q^5)^{-3} +2q^2R(q^5)^{-2} +3q^3R(q^5)^{-1} +5q^4 \\
        &\hspace{-0.65em} -3q^5R(q^5) +2q^6R(q^5)^2 -q^7R(q^5)^3 +q^8R(q^5)^4 \Big),
    \end{aligned}
    \>
which, upon substituting $q^2$ for $q$, yields the 10-dissection
    \<
    \label{eq:phi2R}
    \begin{aligned}
        \fr{1}{f_2} = \fr{f_{50}^5}{f_{10}^6} \Big(& R(q^{10})^{-4} +q^2R(q^{10})^{-3} +2q^4R(q^{10})^{-2} +3q^6R(q^{10})^{-1} +5q^8 \\
    &\qquad\hspace{-0.3em} -3q^{10}R(q^{10}) +2q^{12}R(q^{10})^2 -q^{14}R(q^{10})^3 +q^{16}R(q^{10})^4 \Big).
    \end{aligned}
    \>
We now consider 10-dissections for $X f_{10}^4$ and $Y f_{10}^4$. For $0 \leq k \leq 4$, using the Jacobi triple product formula \eqref{eq:JTP} let
    \begin{align}
        \label{eq:uk}
        u_k(q) &:= (q^{5-k},q^{5+k},q^{10};q^{10})_\infty = \sum_{m=-\infty}^\infty (-1)^m q^{5m^2+km}, \\
        \label{eq:ukdag}
        u_k^{\dag}(q) &:= (-q^{5-k},-q^{5+k},q^{10};q^{10})_\infty = \sum_{m=-\infty}^\infty q^{5m^2+km},
    \end{align}
and
    \<
        \label{eq:Uk}
        U_k(q) := u_k(q^{10}) = (q^{50-10k}, q^{10k+50}, q^{100}; q^{100})_\infty = \sum_{m=-\infty}^\infty (-1)^m q^{50m^2+10km}.
    \>
We note that $u^{\dag}_k(q)$ is generally not the same as $u_k(-q)$. As
    \[
        Xf_{10}^4 = u_1u_2^{\dag}u_3^{\dag}u_4 \qquad\text{and}\qquad Yf_{10}^4 = u_1^{\dag}u_2u_3u_4^{\dag},
    \]
relations \eqref{eq:XYfracs} imply that
    \<
        \label{eq:XYinv-u}
        \nfr{f_{10}^3}{X} = \nfr{u_1^{\dag}u_2u_3u_4^{\dag}\,}{f_2}, \qquad\text{and}\qquad \nfr{f_{10}^3}{Y} = \nfr{u_1u_2^{\dag}u_3^{\dag}u_4\,}{f_2}.
    \>
From the definitions of $R(q)$ and $U_k(q)$, we further observe that
    \<
        \label{eq:RU}
        R(q^{10}) = \fr{(q^{10},q^{40};q^{50})_\infty}{(q^{20},q^{30};q^{50})_\infty} = \fr{U_1(q)U_4(q)}{U_2(q)U_3(q)}.
    \>

In the remainder of our discussions, it is convenient to let
    \<
        \cR := R(q^{10}) \qquad\text{and}\qquad \xF := \nfr{f_{50}^2}{f_{100}}.
    \>
Considering \eqref{eq:XYinv-u}, we now compute 10-dissections for certain quantities of the form $u_i u_j^\dagger$.

\begin{lemma}
\label{lem:uiDissections}
    One has
        \begin{align}
            \label{eq:u1-u3dgV}
            u_1u_3^{\dag} &= \lf(U_1 +q^2U_1\cR -q^4U_3 -2q^6U_2U_3\xF^{-1} + 0q^8 \rh)\xF, \\
            \label{eq:u1dg-u3V}
            u_1^{\dag}u_3 &= \lf(U_1 -q^2U_3\cR^{-1} +q^4U_3 +0q^6 - 2q^8U_1U_4\xF^{-1} \rh)\xF,
        \end{align}
    and
        \begin{align}
        \label{eq:u2-u4dgV}
        &\begin{aligned}
            u_2u_4^{\dag} &= U_1^2 +qU_2\xF + 0q^2 -q^3U_2^2 -q^4U_1U_3 -q^{15}U_4^2 +0q^6 -q^7U_4\xF \\
            &\qquad  -q^8U_3^2 +q^9U_2U_4,
        \end{aligned} \\
        \label{eq:u2dg-u4V}
        &\begin{aligned}
            u_2^{\dag}u_4 &= U_1^2 -qU_2\xF +0q^2 +q^3U_2^2 -q^4U_1U_3 +q^{15}U_4^2 +0q^6 +q^7U_4\xF \\
            &\qquad  -q^8U_3^2 -q^9U_2U_4.
        \end{aligned}     
        \end{align} 
\end{lemma}

\begin{remark}
    \label{rem:5dissection}
    In the following proof, our derivation of the 10-dissection \eqref{eq:u1-u3dgV} appears no differently than the derivation of a 5-dissection. This is because 
        \[
            u_1u_3^{\dag} = (-q^{2},q^{4},q^{6},-q^{8},q^{10},q^{10};q^{10})_\infty
        \]
    only involves even powers of $q$ to begin with, and therefore we need only consider the five components of types 0, 2, 4, 6, and 8 of $u_1u_3^\dag$.
\end{remark}

\begin{proof}
From the definitions of $u_k$ and $u_k^\dag$ we see that
    \<
        \label{eq:u1u3dagsum}
        u_1u_3^\dag = \sum_{m,n=-\infty}^\infty (-1)^m q^{5m^2+m+5n^2+3n}.
    \>
Loosely speaking, the idea for the 10-dissection of this sum (essentially a 5-dissection, per Remark \ref{rem:5dissection}), is to separate the sum over $m,n \in \zz$ into five subsums, where in each subsum we use the change(s) of variables
    \[
    \begin{aligned}
        m &= \xl_r^{(1)} + 2\mu - \nu \\
        n &= \xl_r^{(2)} + \mu + 2\nu
    \end{aligned}
    \qquad (\text{for $0 \leq r \leq 4$ and $\mu,\nu \in \zz$}),
    \]
where the five pairs $\xl_r = (\xl_r^{(1)},\xl_r^{(2)})$ are selected so that summing over $0 \leq r \leq 4$ and $\mu,\nu \in \zz$ exactly covers the sum over $m,n \in \zz$. To express this in a way that is both formal and succinct, we find it convenient to take a slightly geometric point of view.

Thinking of $\zz^2$ as the standard square lattice in the Euclidean plane, define $g:\zz^2\to\zz$ via $g((m,n))=5m^2+m+5n^2+3n$, and let $\xL \subset \zz^2$ denote the lattice generated by the \emph{row} vectors $(2,1)$ and $(-1,2)$. That is, let
    \[
        \xL = (2,1)\zz + (-1,2)\zz.
    \]
One may easily check that the 5 cosets 
    \<
    \label{eq:Cosets}
    \begin{aligned}
        &\xL_0 := (0,0) + \xL, \quad \xL_1 := (1,0) + \xL, \quad \xL_2 := (0,-1) + \xL,\\
        &\xL_3 := (0,1) + \xL, \quad \text{and} \quad \xL_4 := (-1,0)+\xL
    \end{aligned}
    \>
partition $\zz^2$ and that $g((m,n)) \equiv r \,\,(\mathrm{mod}\, 5)$ on $\xL_r$ for $0 \leq r \leq 4$.
Next, it is convenient to abuse notation slightly and write $g$ as a function on the ``affine'' set $\zz^2 \times \{1\}$, identifying $g((m,n,1))=g((m,n))$. Doing this, by defining $\xl_r = (\xl_r^{(1)},\xl_r^{(2)})$ so that $\xL_r = \xl_r + \xL$ in \eqref{eq:Cosets}, and writing 
    \[
    B_r := \left(\begin{array}{ccc}
        2 & 1 & 0 \\
        -1 & 2 & 0 \\
        \xl_r^{(1)} & \xl_r^{(2)} & 1
    \end{array}\right),
    \]
we may succinctly write the 10-dissection of $u_1u_3^{\dag}$ as
    \<
        \label{eq:u1-u3dgSum}
        u_1u_3^{\dag} = \sum_{(m,n)\in\zz^2} (-1)^m q^{g((m,n,1))} = \sum_{r=0}^4 \sum_{(\mu,\nu) \in \zz^2} (-1)^{2\mu-\nu+\xl_r^{(1)}} q^{g((\mu,\nu,1)B_r)}.
    \>

From all this, one may quickly compute that
    \begin{align*}
        u_1u_3^{\dag} &= \lf(\sum_{\mu,\nu=-\infty}^\infty (-1)^\nu q^{25\mu^2+5\mu+25\nu^2+5\nu}\rh) - q^6\lf(\sum_{\mu,\nu=-\infty}^\infty (-1)^\nu q^{25\mu^2+25\mu+25\nu^2+5\nu}\rh) \\
        & \qquad + q^2\lf(\sum_{\mu,\nu=-\infty}^\infty (-1)^\nu q^{25\mu^2+5\mu+25\nu^2+15\nu}\rh) + q^8\lf(\sum_{\mu,\nu=-\infty}^\infty(-1)^\nu q^{25\mu^2+15\mu+25\nu^2+25\nu}\rh) \\
        & \qquad - q^4\lf(\sum_{\mu,\nu=-\infty}^\infty(-1)^\nu q^{25\mu^2+15\mu+25\nu^2+15\nu}\rh),
    \end{align*}
which is
    \<
    \label{eq:u1-u3dgRaw}
    \begin{aligned}
        &= (-q^{20},q^{20},-q^{30},q^{30},q^{50},q^{50};q^{50})_\infty 
            -q^{6}(-1,q^{20},q^{30},-q^{50},q^{50},q^{50};q^{50})_\infty \\
        &\qquad +q^2(q^{10},-q^{20},-q^{30},q^{40},q^{50},q^{50};q^{50})_\infty  +q^{8}(1,-q^{10},-q^{40},q^{50},q^{50};q^{50})_\infty \\
        &\qquad -q^4(-q^{10},q^{10},-q^{40},q^{40},q^{50},q^{50};q^{50})_\infty.
    \end{aligned}
    \>
Next, recalling that $\cR = R(q^{10}) = U_1U_4/U_2U_3$, in \eqref{eq:u1-u3dgRaw} we observe that the coefficient of $q^2$ is
    \begin{align*}
        &(q^{10},-q^{20},-q^{30},q^{40},q^{50},q^{50};q^{50})_\infty = \fr{(q^{10},q^{40},q^{40},q^{60},q^{60},q^{90};q^{100})_\infty}{(q^{20},q^{30};q^{50})_\infty} \lf(\fr{f_{50}^2f_{100}^3}{f_{100}^3}\rh) \\
        &\qquad = \fr{U_1^2U_4}{U_2U_3} \lf(\fr{f_{50}^2}{f_{100}}\rh) = U_1\cR \xF,
    \end{align*}
and (since $(-1;q^{50})_\infty = 2(-q^{50};q^{50})_\infty$) that the coefficient of $-q^{6}$ is
    \[
        (-1,q^{20},q^{30},-q^{50},q^{50},q^{50};q^{50})_\infty = 2U_2U_3.
    \]
Similarly considering the other terms in \eqref{eq:u1-u3dgRaw}, we deduce that
    \[
        u_1u_3^{\dag} = U_1\xF + q^2U_1\xF\cR - q^4U_3\xF - 2q^6U_2U_3 + 0q^8,
    \]
which is formula \eqref{eq:u1-u3dgV}. Equation \eqref{eq:u1dg-u3V} is established by first replacing $(-1)^m$ by $(-1)^n$ in equation \eqref{eq:u1u3dagsum} and then simply repeating our arguments using the same $\xL$, $\xl_r$, and $B_r$.

To derive equation \eqref{eq:u2-u4dgV}, we use a setup similar to the one above. Namely, we write
    \[
        u_2 u_4^{\dag} = \sum_{m,n = -\infty}^\infty (-1)^m q^{h((m,n,1))},
    \]
where
    \[
        h((m,n,1)) := 5m^2 + 2m + 5n^2 + 4n,
    \]
and this time we define
    \[
        \xL = (3,1)\zz + (-1,3)\zz, \qquad \Lambda_r = \lambda_r + \Lambda,
    \]
    \[
    \begin{alignedat}{5}
        &\xl_0 = (0,0), \quad &&\xl_1 = (0,-1), \quad &&\xl_2 = (-1,1), \quad &&\xl_3 = (-1,0), \quad &&\xl_4 = (-1,-1), \\
        &\xl_5 = (-1,-2), \quad &&\xl_6 = (1,1), \quad &&\xl_7 = (1,0), \quad &&\xl_8 = (1,-1), \quad &&\xl_9 = (0,1),
    \end{alignedat}
    \]
and
    \[
        B_r := \left(\begin{array}{ccc}
            3 & 1 & 0 \\
            -1 & 3 & 0 \\
            \xl_r^{(1)} & \xl_r^{(2)} & 1
        \end{array}\right) \qquad (r=0,1,2,\ldots,9).
    \]
Defined this way, we have $h((m,n,1)) \equiv r \mmod{10}$ on $\Lambda_r$, and similar to equation \eqref{eq:u1-u3dgSum} we now have
    \[
        u_2 u_4^{\dag} = \sum_{r=0}^9 \sum_{(\mu,\nu)\in\zz^2} (-1)^{3\mu-\nu+\xl_r^{(1)}} q^{h((\mu,\nu,1)B_r)}.
    \]
Tabulating
    \[
    \begin{array}{c|cc}
    r & (-1)^{\xl_r^{(1)}} & h((\mu,\nu,1)B_r) \\[0.2em]
    \hline && \\[-0.8em]
    0 & +1 & 0+50\mu^2+10\mu +50\nu^2+10\nu \\
    1 & +1 & 1+50\mu^2+0\mu +50\nu^2-20\nu \\
    2 & -1 & 12+50\mu^2-10\mu +50\nu^2+50\nu \\
    3 & -1 & 3+50\mu^2-20\mu +50\nu^2+20\nu \\
    4 & -1 & 4+50\mu^2-30\mu +50\nu^2-10\nu \\
    5 & -1 & 15+50\mu^2-40\mu +50\nu^2-40\nu \\
    6 & -1 & 16+50\mu^2+50\mu +50\nu^2+30\nu \\
    7 & -1 & 7+50\mu^2+40\mu +50\nu^2+0\nu \\
    8 & -1 & 8+50\mu^2+30\mu +50\nu^2-30\nu \\
    9 & +1 & 9+50\mu^2+20\mu +50\nu^2+40\nu, \\
    \end{array}
    \]
we may at once read off the formula
    \[
    u_2u_4^{\dag} = U_1^2 +qU_0U_2 -q^{12}U_1U_5 -q^3U_2^2 -q^4U_1U_3 -q^{15}U_4^2 -q^{16}U_3U_5 -q^7U_0U_4 -q^8U_3^2 +q^9U_2U_4.
    \]
Equation \eqref{eq:u2-u4dgV} then follows, since
\[
    U_0 = (q^{50},q^{50},q^{100};q^{100})_\infty = \fr{f_{50}^2}{f_{100}} = \xF \qquad\text{and}\qquad U_5 = (1,q^{100},q^{100};q^{100})_\infty = 0.
\]
Finally, equation \eqref{eq:u2dg-u4V} is immediate from \eqref{eq:u2-u4dgV} since $u_2^{\dag}(q)u_4(q)=u_2(-q)u_4^{\dag}(-q)$.
\end{proof}

Last before our explicit considerations of $[X^{-1}]_{[2]}$ and $[Y^{-1}]_{[6]}$, we record two identities very useful in our technical computations.

\begin{lemma}
\label{lem:Udiff}
    One has
    \begin{align}
        \label{eq:VRminus}
        U_1 \xF \cR &= U_1U_2-q^{10}U_3U_4,\\
        \label{eq:VRplus}
        U_3 \xF \cR^{-1} &= U_1U_2+q^{10}U_3U_4.
    \end{align}
\end{lemma}

\begin{proof}
    We first recall \cite{hirschhorn:powerofq}*{eqn.~(41.2.1)} and \cite{hirschhorn:powerofq}*{eqn.~(41.2.2)}, which state that
    \[
        (q,-q^2,-q^3,q^4,q^5,q^5;q^5)_\infty = (q^3,q^4,q^6,q^7,q^{10},q^{10};q^{10})_\infty - q(q^1,q^2,q^8,q^9,q^{10},q^{10};q^{10})_\infty,
    \]
    \[
        (-q,q^2,q^3,-q^4,q^5,q^5;q^5)_\infty = (q^3,q^4,q^6,q^7,q^{10},q^{10};q^{10})_\infty + q(q^1,q^2,q^8,q^9,q^{10},q^{10};q^{10})_\infty,
    \]
respectively; in our notation these become
    \[
        X f^2_5 = u_1u_2 - qu_3u_4  \qquad\text{and}\qquad Y f^2_5 = u_1u_2 + qu_3u_4,
    \]
respectively. Since
    \[
        X =(q^1,-q^2,-q^3,q^4;q^5)_\infty = (q^4,q^6;q^{10})_\infty \cdot \fr{(q^1,q^4;q^5)_\infty}{(q^2,q^3;q^5)_\infty} = \fr{u_1R}{f_{10}},
    \]
    \[
        Y =(-q^1,q^2,q^3,-q^4;q^5)_\infty = (q^2,q^8;q^{10})_\infty \cdot \fr{(q^2,q^3;q^5)_\infty}{(q^1,q^4;q^5)_\infty} = \fr{u_3R^{-1}}{f_{10}},
    \]
we deduce that
    \[
        \fr{u_1f^2_5R}{f_{10}} = u_1u_2 - qu_3u_4  \qquad\text{and}\qquad \fr{u_3f^2_5R^{-1}}{f_{10}} = u_1u_2 + qu_3u_4,
    \]    
which are equations \eqref{eq:VRminus} and \eqref{eq:VRplus}, respectively, with $q$ replaced by $q^{1/10}$. 
\end{proof}

\section{The proof of Theorem \ref{thm:X2Y6}}

We recall that our goal is to show that 
    \<
        \label{eq:X2Y60}
        [X^{-1}]_{[2]} = [Y^{-1}]_{[6]} = 0,
    \>
and to this end we let 
    \<
        \label{eq:XXYYdef}
        \fX_{[2]} := \lf[ \fr{f_{10}^9}{Xf_{50}^5} \rh]_{[2]}  \qquad\text{and}\qquad \fY_{[6]} := \lf[ \fr{f_{10}^9}{Y f_{50}^5} \rh]_{[6]}.
    \>
Since $[g(q)h(q^{10})]_{[r]} = h(q^{10})[g(q)]_{[r]}$ for any $q$-series $g(q)$ and $h(q)$, we have
    \[
        \fX_{[2]} = \fr{f_{10}^9}{f_{50}^5} [X^{-1}]_{[2]} \qquad\text{and}\qquad \fY_{[6]} = \fr{f_{10}^9}{f_{50}^5} [Y^{-1}]_{[6]},
    \]
and thus to establish \eqref{eq:X2Y60} it suffices to show that $\fX_{[2]} = \fY_{[6]} = 0$. 
As we have
    \<
    \label{eq:fX2-uAndf}
        \fX_{[2]} = \lf[Yf_{10}^4 \cdot \fr{f_{10}^{6}}{f_2 f_{50}^5}\rh]_{[2]} = \lf[ u_1^{\dag}u_2u_3u_4^{\dag} \cdot \fr{f_{10}^6}{f_2 f_{50}} \rh]_{[2]},
    \>
and similarly for $\fY_{[6]}$ (after replacing $u_1^{\dag}u_2u_3u_4^{\dag}$ with $u_1u_2^{\dag}u_3^{\dag}u_4$), our plan is to use the 10-dissections computed in section \ref{sec:dissections} to compute $\fX_{[2]}$ and $\fY_{[6]}$ and then use various $q$-series identities to reduce these both to 0.

Thus, the proof of Theorem \ref{thm:X2Y6} is now reduced to an exercise in symbolic manipulation, so for the sake of legibility we recall that 
    \[
        \xF = f_{50}^2 / f_{100} \qquad\text{and}\qquad \cR = R(q^{10}), 
    \]
and we write
    \[
        \xa=U_1, \quad \xb=U_2, \quad \xg=U_3, \quad\text{and}\quad \xd=U_4.
    \]
With this notation, equation \eqref{eq:phi2R} states that
    \<
    \label{eq:phi2Rrecall}
    \begin{aligned}
        \fr{f_{10}^6}{f_2 f_{50}^5} = \cR^{-4} + q^2 \cR^{-3} + 2q^4\cR^{-2} + 3q^6\cR^{-1} + 5q^8 -3q^{10}\cR + 2q^{12}\cR^{2} -q^{14}\cR^3 + q^{16}\cR^4,
    \end{aligned}
    \>
and Lemma \ref{lem:uiDissections} states that
    \begin{align}
        \label{eq:u1-u3dg}
        u_1u_3^{\dag} &= \lf(\xa+q^2\xa \cR - q^4\xg - 2q^6\xb\xg \xF^{-1} + 0q^8 \rh) \xF, \\
        \label{eq:u1dg-u3}
        u_1^{\dag}u_3 &= \lf(\xa - q^2\xg\cR^{-1} + q^4\xg + 0q^6 - 2q^8\xa\xd\xF^{-1} \rh)\xF,
    \end{align}
and
    \begin{align}
        \label{eq:u2-u4dg}
        u_2u_4^{\dag} &= \xa^2 + q^1\xb\xF + 0q^2 - q^3\xb^2 - q^4 \xa\xg - q^{15}\xd^2 + 0q^6 - q^7\xd\xF - q^8\xg^2 + q^9\xb\xd,\\
        \label{eq:u2dg-u4}
        u_2^{\dag}u_4 &= \xa^2 - q^1\xb\xF + 0q^2 + q^3\xb^2 - q^4\xa\xg + q^{15}\xd^2 + 0q^6 + q^7\xd\xF - q^8\xg^2 - q^9\xb\xd.
    \end{align} 
Similarly, equations \eqref{eq:VRplus} and \eqref{eq:VRminus} of Lemma \ref{lem:Udiff} are now written as
    \<
        \label{eq:Rpmalpha}
        \xa\xF\cR = \xa\xb-q^{10}\xg\xd \qquad\text{and}\qquad \xg\xF\cR^{-1} = \xa\xb+q^{10}\xg\xd,
    \>
respectively, and equation \eqref{eq:RU} is written as
    \<
        \label{eq:Ralpha}
        \cR = \nfr{\xa\xd}{\xb\xg}.
    \>
In addition, we have the following immediate but important corollary to Lemma \ref{lem:Udiff}.

\begin{corollary}
    \label{cor:Adiff0}
    One has
        \<
            \label{eq:diff0}
            \xa^3\xd^2(\xa\xb + q^{10}\xg\xd) - \xb^2\xg^3(\xa\xb - q^{10}\xg\xd) = 0.
        \>
\end{corollary}

\begin{proof}
    By applying \eqref{eq:Rpmalpha} and then \eqref{eq:Ralpha}, the left-hand side of \eqref{eq:diff0} is immediately seen to be $\xa^3\xd^2\big(\fr{\xg^2\xb\xF}{\xa\xd}\big) - \xb^2\xg^3\big(\fr{\xa^2\xd\xF}{\xb\xg}\big) = 0$.
\end{proof}

\begin{remark}[An outline of our methods]
\label{rem:Method}
    Unfortunately, as may be seen across various $q$-series papers (see, e.g., \cites{hirschhorn:remarkable,chern:Vanishing5,mclaughlin:vanishing}), there is no apparent ``uniform'' method for simplifying large $q$-series expressions that we may employ here. However, we are able to outline the common ``algorithm'' we use in proving Theorems \ref{thm:X2Y6} and \ref{thm:XY8-XY0}.
    \begin{enumerate}
        \item \label{prin:0} We first use formulae \eqref{eq:phi2Rrecall}--\eqref{eq:u2dg-u4} to compute expressions for $\fX_{[2]}$ and $\fY_{[6]}$. Because we are dealing with type-10 components, and in fact showing that $\fX_{[2]} = \fY_{[6]} = 0$, we may freely multiply our expressions by powers of $\cR=R(q^{10})$ to facilitate our use of \eqref{eq:Rpmalpha}.
        \item \label{prin:1} Next, we use relations \eqref{eq:Rpmalpha} to eliminate $\xF$ from our expressions.  
        \item \label{prin:2} Once $\xF$ is eliminated, we eliminate $\cR$ using the identity $\cR=\xa\xd/\xb\xg$.
        \item \label{prin:3} Lastly, we simplify and factor our expressions, and locate a factor of 
            \[ 
                \xa^3\xd^2(\xa\xb + q^{10}\xg\xd) - \xb^2\xg^3(\xa\xb - q^{10}\xg\xd),
            \]
        which is 0 by Corollary \ref{cor:Adiff0}.
    \end{enumerate}
\end{remark}

We now start our consideration of $\fX_{[2]}$. First, using equations \eqref{eq:phi2Rrecall}--\eqref{eq:u2dg-u4} and then multiplying by $\cR^{-2}$ for convenience, we have the chain of equalities
    \begin{align}
    \label{eq:XXp2}
    \begin{aligned}
        \fX_{[2]}\cR^{-2} &= 
            \lf[ Y f_{10}^4 \cdot \fr{f_{10}^6}{f_2f_{50}^5}\rh]_{[2]}\cR^{-2}
            = \lf[ u_1^{\dag}u_2u_3u_4^{\dag} \cdot \fr{f_{10}^6}{f_2f_{50}^5} \rh]_{[2]}\cR^{-2} \\
        &= 2q^{10}\xa^2\xg\xd\cR^{-6}-4q^{10}\xa^3\xd\cR^{-4} +6q^{20}\xa\xg^2\xd\cR^{-3} -6q^{20}\xa^2\xg\xd\cR^{-1} \\
        &\qquad +2q^{20}\xa^3\xd\cR +2q^{30}\xa\xg^2\xd\cR^2 +\xa^3\xF\cR^{-5} +2q^{10}\xa^3\xF +(3q^{20}\xg^2)\xa\xF\cR \\
        &\qquad +\lf(-\xa^2\cR^{-6} -q^{10}\xa\xg\cR^{-3} +3q^{10}\xa^2\cR^{-1} +5q^{20}\xg^2\rh)\xg\xF\cR^{-1}.
    \end{aligned}
    \end{align}
Next, using \eqref{eq:Rpmalpha} changes this to
    \<
    \label{eq:XXp3}
    \begin{aligned}
        \fX_{[2]}\cR^{-2} 
        &= -\xa^3\xb\cR^{-6} +q^{10}\xa^2\xg\xd\cR^{-6} -4q^{10}\xa^3\xd\cR^{-4} -q^{10}\xa^2\xb\xg\cR^{-3} +5q^{20}\xa\xg^2\xd\cR^{-3}  \\
        &\qquad +3q^{10}\xa^3\xb\cR^{-1} -3q^{20}\xa^2\xg\xd\cR^{-1} +8q^{20}\xa\xb\xg^2 +2q^{30}\xg^3\xd +2q^{20}\xa^3\xd\cR \\
        &\qquad +2q^{30}\xa\xg^2\xd\cR^2 +\xa^3\xF\cR^{-5} +2q^{10}\xa^3\xF.
    \end{aligned}
    \>
Multiplying \eqref{eq:XXp3} by $\cR^6$ and identifying
    \[
        \xa^3\xF\cR + 2q^{10}\xa^3\xF\cR^6  = \xa^3\xb - q^{10}\xa^2\xg\xd + 2q^{10}\xa^3\xb\cR^5 - 2q^{20}\xa^2\xg\xd\cR^5,
    \]
we further have
    \<
    \begin{aligned}
        \fX_{[2]}\cR^{4} &= -4q^{10}\xa^3\xd\cR^2 -q^{10}\xa^2\xb\xg\cR^3 +5q^{20}\xa\xg^2\xd\cR^3 +5q^{10}\xa^3\xb\cR^5  -5q^{20}\xa^2\xg\xd\cR^5 \\
        &\qquad +8q^{20}\xa\xb\xg^2\cR^6 +2q^{30}\xg^3\xd\cR^6 +2q^{20}\xa^3\xd\cR^7 +2q^{30}\xa\xg^2\xd\cR^8.
    \end{aligned}
    \>
Using the relation $\cR=\xa\xd/\xb\xg$ to eliminate $\cR$ from the right-hand side here and then simplifying, we find (or at least may verify) that
    \[
        \fX_{[2]}\cR^4 = \Big(\xa^3\xd^2(\xa\xb + q^{10}\xg\xd) - \xb^2\xg^3(\xa\xb - q^{10}\xg\xd)\Big)\left( \fr{5q^{10}\xa^4\xd^3}{\xb^5\xg^5} + \fr{2q^{20}\xa^6\xd^6}{\xb^8\xg^7} \right),
    \]
which by Corollary \ref{cor:Adiff0} implies that $\fX_{[2]}=0$, as claimed.

We now turn to $\fY_{[6]}$. Similar to our derivation of \eqref{eq:XXp2}, we again use equations \eqref{eq:phi2Rrecall}--\eqref{eq:u2dg-u4} and then multiply by $\cR^{-4}$ for convenience to find that
    \[
    \begin{aligned}
        \fY_{[6]}\cR^{-4} 
        &= \lf[ Xf_{10}^4 \cdot \fr{f_{10}^6}{f_2f_{50}^5} \rh]_{[6]}\cR^{-4} =\lf[ u_1u_2^{\dag}u_3^{\dag}u_4 \cdot \fr{f_{10}^6}{f_2f_{50}^5} \rh]_{[6]}\cR^{-4} \\
        &= -2\xa^2\xb\xg\cR^{-8} +2q^{10}\xb\xg^3\cR^{-7} +6q^{10}\xa\xb\xg^2\cR^{-5} +6q^{10}\xa^2\xb\xg\cR^{-3} \\
        &\qquad  +4q^{20}\xb\xg^3\cR^{-2} +2q^{20}\xa\xb\xg^2 +5\xa^3\xF\cR^{-5} -(q^{20}\xg^2)\xa\xF\cR \\
        &\qquad + \lf( -3\xa^2\cR^{-6} +2q^{10}\xg^2\cR^{-5} -3q^{10}\xa\xg\cR^{-3} -q^{10}\xa^2\cR^{-1} -q^{20}\xg^2 \rh)\xg\xF\cR^{-1}.
        \end{aligned}
    \]
Next, using \eqref{eq:Rpmalpha} changes this to
\begin{align}
\label{eq:YYp3}
    &\begin{aligned}
        \fY_{[6]}\cR^{-4} &= -2\xa^2\xb\xg\cR^{-8} +2q^{10}\xb\xg^3\cR^{-7} -3\xa^3\xb\cR^{-6} -3q^{10}\xa^2\xg\xd\cR^{-6} \\
        &\qquad +8q^{10}\xa\xb\xg^2\cR^{-5} +2q^{20}\xg^3\xd\cR^{-5} +3q^{10}\xa^2\xb\xg\cR^{-3} -3q^{20}\xa\xg^2\xd\cR^{-3} \\
        &\qquad +4q^{20}\xb\xg^3\cR^{-2} -q^{10}\xa^3\xb\cR^{-1} -q^{20}\xa^2\xg\xd\cR^{-1} +5\xa^3\xF\cR^{-5}.
    \end{aligned}
\end{align}
Multiplying \eqref{eq:YYp3} by $\cR^{6}$ and identifying 
    \[
        5\xa^3\xF\cR = 5\xa^3\xb -5q^{10}\xa^2\xg\xd,
    \]
we have
    \[
    \begin{aligned}
        \fY_{[6]}\cR^2 &= -2\xa^2\xb\xg\cR^{-2} +2q^{10}\xb\xg^3\cR^{-1} +2\xa^3\xb -8q^{10}\xa^2\xg\xd +8q^{10}\xa\xb\xg^2\cR +2q^{20}\xg^3\xd\cR \\
        &\qquad  +3q^{10}\xa^2\xb\xg\cR^3 -3q^{20}\xa\xg^2\xd\cR^3 +4q^{20}\xb\xg^3\cR^4 -q^{10}\xa^3\xb\cR^5 -q^{20}\xa^2\xg\xd\cR^5.
    \end{aligned}
    \]
Again using the equality $\cR=\xa\xd/\xb\xg$ to eliminate $\cR$ from the right-hand side here, and then simplifying, yields the relation
    \[
        \fY_{[6]}\cR^2 = \Big( \xa^3\xd^2(\xa\xb +q^{10}\xg\xd)-\xb^2\xg^3(\xa\xb -q^{10}\xg\xd) \Big) \lf(\fr{2}{\xa\xd^2} +\fr{2q^{10}\xa\xd}{\xb^3\xg^2} -\fr{q^{10}\xa^4\xd^3}{\xb^5\xg^5}\rh).
    \]
Using Corollary \ref{cor:Adiff0} we deduce that $\fY_{[6]} = 0$, which concludes the proof of Theorem \ref{thm:X2Y6}.

\section{The proof of Theorem \ref{thm:XY8-XY0}}
Establishing that $\mathfrak{p}(10j,\chi_5^\dagger) = \mathfrak{p}(10j,\chi_5)$ and $\mathfrak{p}(10j+8,\chi_5^\dagger) = -\mathfrak{p}(10j+8,\chi_5)$ for $j \geq 0$, or equivalently that
    \[
        [Y^{-1}]_{[0]} = [X^{-1}]_{[0]} \qquad\text{and}\qquad [Y^{-1}]_{[8]} = -[X^{-1}]_{[8]}, 
    \]
is done using the same methods as in the proof of Theorem \ref{thm:X2Y6}, so we merely outline proofs below. Akin to the definitions \eqref{eq:XXYYdef} of $\fX_{[2]}$ and $\fY_{[6]}$ we define
    \<
        \fX_{[0]} = \lf[ \fr{f_{10}^9}{Xf_{50}^5} \rh]_{[0]} \qquad\text{and}\qquad \fY_{[0]} = \lf[ \fr{f_{10}^9}{Y f_{50}^5} \rh]_{[0]},
    \>
and similarly for $\fX_{[8]}$ and $\fY_{[8]}$. 

Again using \eqref{eq:phi2Rrecall}--\eqref{eq:u2dg-u4} we find that
    \[
    \begin{aligned}
        &\fX_{[0]} -\fY_{[0]} \\
        &\qquad = -2q^{10}\xa\xb\xg^2\cR^{-4} -2q^{10}\xa^3\xd\cR^{-3} +4q^{10}\xa^2\xb\xg\cR^{-2} +4q^{20}\xa\xg^2\xd\cR^{-2} -6q^{20}\xb\xg^3\cR^{-1}\\
        &\qquad\qquad +10q^{20}\xa^2\xg\xd +6q^{20}\xa\xb\xg^2\cR -4q^{20}\xa^3\xd\cR^2 -2q^{20}\xa^2\xb\xg\cR^3 -2q^{30}\xa\xg^2\xd\cR^3 \\
        &\qquad\qquad -2q^{30}\xb\xg^3\cR^4 + (q^{10}\xg^2\cR^{-4} +q^{10}\xa\xg\cR^{-2} +3q^{10}\xa^2)\xg\xF\cR^{-1} -13q^{20}\xg^3\xF \\
        &\qquad\qquad - (5q^{10}\xa^2 +8q^{20}\xg^2\cR -q^{20}\xa\xg\cR^3)\xa\xF\cR,
    \end{aligned}
    \]
after which the ``steps'' from Remark \ref{rem:Method} directly lead us to the equality
    \[
        (\fX_{[0]} -\fY_{[0]})\cR^{-1} = \Big(\xa^3\xd^2(\xa\xb + q^{10}\xg\xd) -\xb^2\xg^3(\xa\xb -q^{10}\xg\xd)\Big) \lf(\fr{\xb^3\xg^4}{\xa^5\xd^5} -\fr{5q^{10}}{\xb^2\xg} -\fr{2\xb\xg}{\xa^2\xd^3}\rh).
    \]
Using Corollary \ref{cor:Adiff0} we conclude that $\fX_{[0]} = \fY_{[0]}$, and it follows that $[Y^{-1}]_{[0]} = [X^{-1}]_{[0]}$. 

Our proof that $[Y^{-1}]_{[8]} = -[X^{-1}]_{[8]}$ proceeds similarly. Namely, we first find that
    \<
    \begin{aligned}
        &\fX_{[8]}+\fY_{[8]} \\
        &\qquad = -2\xa^3\xd\cR^{-4} -2\xa^2\xb\xg\cR^{-3} +2q^{10}\xa\xg^2\xd\cR^{-3} +4q^{10}\xb\xg^3\cR^{-2} +6q^{10}\xa^2\xg\xd\cR^{-1} \\
        &\qquad\qquad +10q^{10}\xa\xb\xg^2 +6q^{10}\xa^3\xd\cR -4q^{10}\xa^2\xb\xg\cR^2 +4q^{20}\xa\xg^2\xd\cR^2 -2q^{20}\xb\xg^3\cR^3 \\
        &\qquad\qquad +2q^{20}\xa^2\xg\xd\cR^4 - (\xa\xg\cR^{-3} +8\xa^2\cR^{-1} -5q^{10}\xg^2)\xg\xF\cR^{-1} +13\xa^3\Phi \\
        &\qquad\qquad + (3q^{10}\xg^2 -q^{10}\xa\xg\cR^2 +q^{10}\xa^2\cR^4)\xa\xF\cR,
    \end{aligned}
    \notag
    \>
after which the ``steps'' of Remark \ref{rem:Method} lead us to the relation
    \[
        (\fX_{[8]}+\fY_{[8]})\cR = \Big(\xa^3\xd^2(\xa\xb + q^{10}\xg\xd) - \xb^2\xg^3(\xa\xb - q^{10}\xg\xd)\Big) \lf(\fr{5}{\xa\xd^2} + \fr{2q^{10}\xa\xd}{\xb^3\xg^2} + \fr{q^{10}\xa^4\xd^3}{\xb^5\xg^5}\rh).
    \]
Again by Corollary \ref{cor:Adiff0} it follows that $\fY_{[8]} = -\fX_{[8]}$ and that $[Y^{-1}]_{[8]} = -[X^{-1}]_{[8]}$, which completes the proof of Theorem \ref{thm:XY8-XY0}.

\bibliography{vanishing.bib}
\end{document}